\documentclass[12pt]{amsart}

\usepackage{amssymb,amsthm}

\newtheorem{theorem}{Theorem}[section]

\newtheorem{lemma}[theorem]{Lemma}

\def\irr#1{{\rm  Irr}(#1)}

\newcommand{\QQ}{{\mathbb Q}}

\begin{document}

\title[Solvable groups]{Landau's theorem, fields of values for characters, and solvable groups}
\author[Mark L. Lewis]{Mark L. Lewis}

\address{Department of Mathematical Sciences, Kent State University, Kent, OH 44242}
\email{lewis@math.kent.edu}


\begin{abstract}
When $G$ is solvable group, we prove that the number of conjugacy classes of elements of prime power order is less than or equal to the number of irreducible characters with values in fields where $\QQ$ is extended by prime power roots of unity.  We then combine this result with a theorem of H\'ethelyi and K\"ulshammer that bounds the order of a finite group in terms of the number of conjugacy classes of elements of prime power order to bound the order of a solvable group by the number of irreducible characters with values in fields extended by prime power roots of unity.  This yields for solvable groups a generalization of Landau's theorem.
\end{abstract}

\maketitle

\section{Introduction}

In this note, all groups are finite.  The famous theorem of Landau states that the order of a group $G$ can be bounded in terms of a function of the number of its conjugacy classes.  In Theorem 1.1 of \cite{HeKu}, H\'ethelyi and K\"ulshammer showed that in fact the order of the group is bounded in terms of the number of its conjugacy classes of elements of prime power order.

It is well known that there is a duality between conjugacy classes and ordinary characters.  In particular, the number of conjugacy classes equals the number of irreducible characters, so Landau's theorem could equivalently be stated as saying that if $G$ is a group, then $|G|$ is bounded by a function in terms of $|\irr G|$.  Thus, it makes sense to ask if there is a character theoretic version of the H\'ethelyi - K\"ulshammer theorem.  We will show that there is for solvable groups.

If $p$ is a prime, we define $\QQ_p$ to be the field obtained by adjoining all $p^a$th roots of unity for all positive integers $a$ to $\QQ$.  We say that a character is $\QQ_{pp}$-valued if there is a prime $p$ so that the values of the character all lie in $\QQ_p$.  With this definition, we prove the following.

\begin{theorem}\label{main a}
For every positive integer $k$, there exist a finite number of solvable groups with exactly $k$ $\QQ_{pp}$-valued irreducible characters.
\end{theorem}

In fact, we will show that Theorem \ref{main a} is a consequence of the Theorem of H\'ethelyi and K\"ulshammer because of the following fact.

\begin{theorem}\label{main b}
If $G$ is a solvable group, then the number of prime power conjugacy classes in $G$ is less than or equal to the number of $\QQ_{pp}$-valued irreducible characters of $G$.
\end{theorem}

Following up on the parallelism between conjugacy classes and irreducible characters, it is tempting to believe that the number of prime power conjugacy classes should equal the number of $\QQ_{pp}$ valued irreducible characters of $G$, but we will present a number of examples with these two sets do not have equal sizes.  However, if we restrict ourselves to groups of odd order, then we do obtain equality.

\begin{theorem}\label{main odd}
If $|G|$ is odd, then the number of prime power conjugacy classes in $G$ equals the number of $\QQ_{pp}$-valued irreducible characters of $G$.
\end{theorem}

We note that the Theorem of H\'ethelyi and K\"ulshammer is proved for all finite groups, and unfortunately, our results are only for solvable groups.  Unfortunately, the key to our proof relies on facts that are true only for solvable groups.  However, we believe that one should be able to remove the solvability hypothesis from Theorem \ref{main a}, but different tools will need to be developed to do this.  On the other hand, it is tempting to ask whether or not Theorem \ref{main b} is true if the solvability hypothesis is removed.  We know of examples of nonsolvable groups where the number of $\QQ_{pp}$-valued irreducible characters is greater than the number of prime power conjugacy classes, namely any symmetric group of degree $n \ge 5$, and we do not know of any nonsolvable group where the number of $\QQ_{pp}$-valued irreducible characters is less than the number of prime power conjugacy classes, but we have not checked very thoroughly.

We will prove Theorem \ref{main b} one prime at a time.  In fact, we can weaken the hypothesis that $G$ is solvable to $G$ is $p$-solvable for an arbitrary prime $p$.

\begin{theorem} \label{main c}
Let $p$ be a prime and let $G$ be a $p$-solvable group.  Then the number of conjugacy classes of $p$-elements is less than or equal to the number of $\QQ_p$-valued irreducible characters.
\end{theorem}

It is not possible to prove Theorem \ref{main a} one prime at a time.  In particular, we will present examples of solvable groups $|G|$ with two or three irreducible characters having values in $\QQ_p$, but where $|G|_p$ is arbitrarily large.  However, we will show that one can bound the number of chief factors that are $p$-groups for a $p$-solvable group is bounded by the number of $\QQ_p$-valued irreducible characters.

\begin{theorem} \label{main d}
If $G$ is a $p$-solvable group for a prime $p$, then the number of $p$-chief factors in a chief series for $G$ is at most the number of nonprincipal $\QQ_p$-valued irreducible characters of $G$.  Furthermore, if $G$ is a solvable group, then the number of factors in a chief series for $G$ is less than or equal to the number of nonprincipal irreducible $\QQ_{pp}$-valued characters.
\end{theorem}

Note that the nonprincipal $\QQ_p$-valued and $\QQ_{pp}$-valued irreducible characters of $G$ can be read off of the character table for $G$.  We will see that the results of this paper are essentially a direct application of the ${\rm B}_\pi$ characters developed by Professor Isaacs.

We would like to thank Nguyen Ngoc Hung for bringing this problem to my attention and for several helpful conversations while this paper was being written.

\section{One prime at a time}

We first work one prime at a time.  Let $p$ be a prime.  The following theorem is the key to this section.  We note that if we were able to remove the hypothesis that $G$ is $p$-solvable from this theorem, then Nguyen Hung and I would be able to remove the hypothesis that $G$ must be solvable from Theorem \ref{main a}.

\begin{theorem} \label{main}
Let $G$ be a $p$-solvable group and let $N$ be a minimal normal subgroup of $G$ that is a $p$-group. If $\lambda \in \irr N$, then $\lambda^G$ has an irreducible constituent $\chi$ such that $\chi$ is $\QQ_p$-valued.
\end{theorem}

To prove this theorem, we will use the ${\rm B}_\pi$-characters defined by Isaacs in Definition 5.1 of \cite{pisep}.  Since the definition of ${\rm B}_\pi$-characters is somewhat complicated, we do not repeat it here, but refer the interested reader to \cite{pisep}.  Expository accounts of ${\rm B}_\pi$-characters can be found in \cite{arc} and \cite{aust}.  We let $\pi$ be a set of primes.  When $G$ is a $\pi$-separable group, we write ${\rm B}_{\pi} (G)$ for the set of ${\rm B}_\pi$ characters of $G$.  We note that these characters are only defined when the group is $\pi$-separable, and this is the main barrier we have in removing the solvable hypothesis from Theorems \ref{main a} and \ref{main b}.

The following fact is fairly well-known.

\begin{lemma} \label{key}
Let $\pi$ be a set of primes and let $G$ be a $\pi$-separable group.  If $N$ is normal in $G$ and $\theta \in {\rm B}_{\pi} (G)$, then there exists a character $\chi \in {\rm B}_{\pi} (G)$ such that $\chi$ is a constituent of $\theta^G$.
\end{lemma}

\begin{proof}
We will prove this by induction on $|G:N|$.  If $N = G$, then take $\chi = \theta$, and the conclusion is trivial.  Thus, we assume that $N < G$.  Let $N \le M < G$ so that $M$ is a maximal normal subgroup of $G$.  By the inductive hypothesis, there is a character $\mu \in {\rm B}_\pi (M)$ so that $\mu$ is a constituent of $\theta^M$.  If $G/M$ is a $\pi$-group, then every irreducible constituent of $\mu^G$ lies in ${\rm B}_\pi (G)$ by Theorem 7.1 of \cite{pisep}, so we can take $\chi$ to be any irreducible constituent of $\mu^G$, and we see that $\chi$ will lie in ${\rm B}_\pi (G)$ and is a constituent of $\theta^G$.  Thus, we may assume that $G/M$ is not a $\pi$-group.  Since $G$ is $\pi$-separable, this implies that $G/M$ is a $\pi'$-group.  In this case, we apply Theorem 6.2 (b) of \cite{pisep} to see that $\mu^G$ has some irreducible constituent $\chi$ that lies in ${\rm B}_\pi (G)$.  Since $\chi$ will be a constituent of $\theta^G$, the result follows.
\end{proof}

We can now prove Theorem \ref{main}.

\begin{proof}[Proof of Theorem \ref{main}]
If $\lambda = 1_N$, then take $\chi = 1_G$, and the result holds.  Note that $\lambda$ has degree one and the restriction of $\lambda$ to any subgroup of $N$ has order dividing $p$, so $\lambda$ lies in ${\rm B}_p (N)$ (see Lemma 5.4 of \cite{pisep}).  We now apply Lemma \ref{key} to see that there exists $\chi \in {\rm B}_{p} (G)$ such that $\chi$ is an irreducible constituent of $\lambda^G$.  Finally, we apply Corollary 12.1 of \cite{pisep} to see that $\chi$ is $\QQ_p$-valued.
\end{proof}

We now obtain Theorem \ref{main d} as a corollary.

\begin{proof}[Proof of Theorem \ref{main d}]
We work by induction on $|G|$.  If $G$ is trivial, then the conclusion is trivial.  Thus, we may assume that $G > 1$.  Let $N$ be a minimal normal subgroup of $G$.  By the inductive hypothesis, the number of $p$-chief factors of $G/N$ is at most the number of nonprincipal $\QQ_p$-irreducible characters of $G/N$.  If $N$ is a $p'$-group, this will yield the conclusion.  Thus, we may assume that $N$ is a $p$-group.  In light of Theorem \ref{main}, we see that $G/N$ has at least one fewer nonprincipal $\QQ_p$-irreducible than $G$, and so the first conclusion will hold for $G$.  The second conclusion holds by summing over all the primes that divide $|G|$.
\end{proof}

We say that $g$ is a $p$-element of $G$ if $g \in G$ has $p$-power order.  If $G$ is a $p$-solvable group, then applying Theorem 9.3 of \cite{pisep}, the number of characters in ${\rm B}_p (G)$ equals the number of $p$-conjugacy classes in $G$.  This is key for the following theorem which includes Theorem \ref{main c}.

\begin{theorem} \label{odd}
Let $p$ be a prime and let $G$ be a $p$-solvable group.  Then the number of conjugacy classes of $p$-elements is less than or equal to the number of $\QQ_p$-valued irreducible characters.  Furthermore, if $|G|$ is odd, then equality holds.
\end{theorem}

\begin{proof}
We saw that Corollary 12.1 of \cite{pisep} shows that every character in ${\rm B}_p (G)$ is $Q_p$-valued.  Since the number of $p$-conjugacy classes equals the number of characters in ${\rm B}_p (G)$, this implies that the number of $p$-conjugacy classes is less than or equal to the number of $Q_p$ valued irreducible characters of $G$.  Suppose now that $|G|$ is odd.  Let $E$ be the field obtained by adjoining a primitive $|G|$th root of unity to $Q$.  Following page 551 of \cite{odd}, $E$ has an automorphism $\tau$ that fixes the $p$th power roots of unity and acts like complex conjugation on the $p'$ roots of unity.  Notice that the fixed field for $\tau$ will be $E \cap Q_p$.  Applying Lemma 3.1 of \cite{odd}, we see that if $\chi \in \irr G$, then $\chi \in {\rm B}_p (G)$ if and only if $\chi^\tau = \chi$.  Thus, an irreducible character of $G$ lies in ${\rm B}_p (G)$ if and only if $\tau$ fixes all the values of $\chi$.  Since $E \cap Q_p$ is the fixed field for $\tau$, we conclude that $\chi$ lies in ${\rm B}_p (G)$ if and only if $\chi$ is $Q_p$-valued.  Thus, the number of $p$-conjugacy classes will equal the number of $Q_p$-valued irreducible characters of $G$.
\end{proof}

We note that one cannot bound the order of a Sylow $p$-subgroup of $G$ in terms of a function on the number of $\QQ_p$-valued irreducible characters of $G$.  Fix a prime $p$ and let $n$ be any positive integer.  Let $F$ be the field of order $p^n$ and let $E$ be the additive group of $F$ and let $C$ be the multiplicative group of $F$ so that $|E| = p^n$ and $C = p^n-1$.  It is easy to see that multiplication in $F$ defines a group action of $C$ on $E$, and we let $G$ be the semi-direct resulting from $C$ acting on $E$.  It is not difficult to see that $G$ has a unique faithful irreducible character and that it will be $\QQ_p$-valued.  If $p = 2$, then it is not difficult to see that $1_G$ is the only $\QQ_2$-valued character in $\irr {G/E}$, and so, $G$ has exactly two $\QQ_2$-valued irreducible characters.  On the other hand, if $p$ is odd, then one can see that $\irr {G/E}$ has exactly two $\QQ_p$-valued characters, and so, $\irr G$ will have exactly three irreducible characters that are $\QQ_p$-valued.  Since $n$ is arbitrary, this gives the desired conclusion.

\section{$\QQ_{pp}$-valued irreducible characters.}

This next lemma is also known.

\begin{lemma} \label{empty}
Let $\pi$ and $\rho$ be sets of primes, and let $G$ be a group that is both $\pi$-separable and $\rho$-separable.  If $\pi$ and $\rho$ are disjoint, then ${\rm B}_\pi (G) \cap {\rm B}_\rho (G) = \{ 1_G \}$.
\end{lemma}

\begin{proof}
We work by induction on $|G|$.  If $G = 1$, then the result is trivial.  Thus, we may assume that $G > 1$. Hence, we can find a maximal normal subgroup $N$ in $G$.  Consider a character $\chi \in {\rm B}_\pi (G) \cap {\rm B}_\rho (G)$.  Let $\theta$ be an irreducible constituent of $\chi_N$.    Applying Corollary 7.5 of \cite{pisep}, we know that $\theta$ lies in both ${\rm B}_{\pi} (G)$ and $B_{\rho} (G)$.  By the inductive hypothesis, it follows that $\theta = 1_N$.  Hence, $\chi \in \irr {G/N}$.  Since $G$ is a $\pi$-separable group, we know that $G/N$ is either a $\pi$-group or a $\pi'$-groups.  If $G/N$ is a $\pi$-group, then $G/N$ is a $\rho'$-group since $\pi \cap \rho$ is empty.  It follows that $1_G$ is the unique character in ${\rm B}_{\rho} (G/N)$ (see Corollary 5.3 of \cite{pisep}) and hence, $\chi = 1_G$.  Otherwise, $G/N$ is a $\pi'$-group; so $1_G$ is the unique character in ${\rm B}_{\pi} (G)$ (again Corollary 5.3 of \cite{pisep}), and again, $\chi = 1_G$.
\end{proof}

This next theorem includes both Theorem \ref{main b} and Theorem \ref{main odd}.

\begin{theorem} \label{main2}
If $G$ is a solvable group, then the number of prime power conjugacy classes in $G$ is less than or equal to the number of $\QQ_{pp}$-valued irreducible characters of $G$.  Furthermore, if $|G|$ is odd, then this is an equality.
\end{theorem}

\begin{proof}
We know from that for each prime $p$ the number of $p$-conjugacy classes equals the number of ${\rm B}_p (G)$ characters.  Thus, the number of nonidentity $p$-conjugacy classes equals the number of nonprincipal ${\rm B}_p (G)$ characters.  It is not difficult to see that the nonidentity prime power classes is a disjoint union of the nonidentity $p$-classes for the various primes $p$ that divide $|G|$.  In light of Lemma \ref{empty}, the sets of nonprincipal ${\rm B}_p (G)$ characters are disjoint for the various primes $p$.  We conclude that number of prime power classes equals the size of $\cup_p {\rm B}_p (G)$.  Since all the ${\rm B}_p (G)$ characters are $\QQ_{pp}$-valued for all the primes $p$, we obtain the first conclusion.  Finally, when $|G|$ is odd we have seen that every $\QQ_{pp}$-valued irreducible character of $G$ lies in ${\rm B}_p (G)$ for some prime $p$, and this yields the second conclusion.
\end{proof}

We now present some examples to see that the number of $\QQ_{pp}$-irreducible characters need not equal the number of prime power conjugacy classes.  Recall that a group is called {\it rational} if all the irreducible characters are rational valued.  Thus, if $G$ is a rational group, then all of the irreducible characters will $\QQ_{pp}$-valued, so it suffices to find rational groups that have elements whose orders are not prime powers.  It is not difficult to see that any Symmetric group of degree at least $5$ will fit the bill.  Since we are discussing solvable groups in this note, we feel obligated to present some solvable examples.  Infinitely many solvable group examples will exist since any $\{ 2, 3 \}$-group can be embedded in a rational $\{ 2, 3 \}$-group (see Proposition 1 of \cite{gow}), and of course, every $\{ 2, 3 \}$-group is solvable by Burnside's theorem.

We can now obtain Theorem \ref{main a} as a application of the H\'ethelyi - K\"ulshammer theorem

\begin{proof}[Proof of Theorem \ref{main a}]
In light of Theorem \ref{main2}, we know that $G$ has at most $k$ prime power conjugacy classes.  The H\'ethelyi - K\"ulshammer theorem, Theorem 1.1 of \cite{HeKu}, implies that there are finitely many groups with at most $k$ prime power conjugacy classes.  We conclude that there are finitely many solvable groups with $k$ $\QQ_{pp}$-valued irreducible characters.
\end{proof}

Note that Theorem \ref{main a} is equivalent to saying that if $G$ is a solvable group, then $|G|$ can be bounded in terms of a function in the number of $\QQ_{pp}$-valued irreducible characters of $G$.  Since the number of irreducible characters of $G$ is less than or equal to $|G|$, this implies that $|\irr G|$ can be bounded by a function in terms of the number of $\QQ_{pp}$ irreducible characters.  Also, note that whenever we used $\QQ_p$ or $\QQ_{pp}$ valued irreducible characters, we really were using ${\rm B}_p (G)$ or $\cup_{p \in \pi (G)} {\rm B}_p (G)$ where $\pi (G)$ is the set of primes that divide $|G|$.    Thus, in Theorems \ref{main a}, \ref{main b}, \ref{main c}, and \ref{main d}, we could replace $\QQ_p$ valued irreducible characters by ${\rm B}_p (G)$ and $\QQ_{pp}$ valued irreducible characters by $\cup_{p \in \pi(G)} {\rm B}_p (G)$, and then view all of the results in this paper as results about the ${\rm B}_\pi$ characters of $G$.

\end{document}